\documentclass[a4paper,10pt]{amsart}
\usepackage[utf8]{inputenc}
\usepackage{lmodern}
\usepackage[english]{babel}
\usepackage[T1]{fontenc}
\usepackage{microtype} 
    \usepackage[OT2,OT1]{fontenc}
    \newcommand\cyr{%
    \renewcommand\rmdefault{wncyr}%
    \renewcommand\sfdefault{wncyss}%
    \renewcommand\encodingdefault{OT2}%
    \normalfont
    \selectfont}
    \DeclareTextFontCommand{\textcyr}{\cyr}
\usepackage{amsmath,amssymb,amsfonts,amsthm,amscd,latexsym,mathrsfs}
\usepackage[all]{xy}
\input xypic
\usepackage{color}
\usepackage{hyperref}
\usepackage{lipsum}
\usepackage{breqn}
\hypersetup{colorlinks=true,linkcolor=red,citecolor=blue}
\usepackage[hyperpageref]{backref} 
\usepackage{booktabs}
\usepackage{multirow}
\theoremstyle{plain}
\newtheorem{theorem}[subsection]{{\bf Theorem}}
\newtheorem*{theorem*}{{\bf Theorem}}
\newtheorem{corollary}[subsection]{{\bf Corollary}}
\newtheorem*{corollary*}{{\bf Corollary}}

\newtheorem{lemma}[subsection]{{\bf Lemma}}

\theoremstyle{definition}

\theoremstyle{remark}

\numberwithin{equation}{section}

\begin{document}
\baselineskip=14pt
\title{A note on the order of the Schur multiplier of $p$-groups}
\author[P. K. Rai]{Pradeep K. Rai}
\address[Pradeep K. Rai]{Department of Mathematics, Bar-Ilan University
Ramat Gan \\
Israel}
\email{raipradeepiitb@gmail.com}
\begin{abstract}
Let $G$ be a finite $p$-group of order $p^n$ with $|G'| = p^k$. Let $M(G)$ denotes the Schur multiplier of $G$. A classical result of Green states that $|M(G)| \leq p^{\frac{1}{2}n(n-1)}$. In 2009, Niroomand, improving Green's and other bounds on $|M(G)|$ for a non-abelain $p$-group $G$, proved that $|M(G)| \leq p^{\frac{1}{2}(n-k-1)(n+k-2)+1}$. In this article we note that a bound, obtained earlier, by Ellis and Weigold is more general than the bound of Niroomand. We derive from the bound of Ellis and Weigold that $|M(G)| \leq p^{\frac{1}{2}(d(G)-1)(n+k-2)+1}$ for a non-abelain $p$-group $G$. Moreover, we sharpen the bound of Ellis and Weigold and as a consequence derive that if $G^{ab}$ 
is not homocyclic then $|M(G)| \leq p^{\frac{1}{2}(d(G)-1)(n+k-3)+1}$. We further note an improvement in an old bound given by Vermani. Finally we note, for a $p$-group of coclass $r$,  that $|M(G)| \leq p^{\frac{1}{2}(r^2-r)+kr+1}$. This improves a bound by Moravec.  
\end{abstract}
\maketitle
\section{Introduction}

Let $G$  be a group. The center and the commutator subgroup of $G$ are denoted by $Z(G)$, and $\gamma_2(G)$ respectively. By $d(G)$ we denote the minimal no of generators of $G$. We write $\gamma_i(G)$ and $Z_i(G)$ for the $i$-th term in the lower and upper central series of $G$ respectively. Finally, the abelianization of the group $G$, i.e. $G/\gamma_2(G)$, is denoted by $G^{ab}$. 




\vspace{.2cm}

Let $G$ be finite $p$-group of order $p^n$ and let $M(G)$ denotes the Schur multiplier of $G$. In 1956 Green proved that $|M(G)| \leq p^{\frac{1}{2}n(n-1)}$ \cite{Green}. Since then Green's bound has been reproved and generalized by many mathematicians. Weigold, in 1965, gave a bound on $|\gamma_2(G)|$ in terms of $|G/Z(G)|$ and rederived the Green's bound using the existence of representation groups \cite{Weigold1}. In 1967 Gasch\"{u}tz et al., sharpening Green's bound, proved in \cite{Gaschutz} that 
\begin{equation*}
|M(G)| \leq |M(G^{ab})||\gamma_2(G)|^{d(G/Z(G))-1}. \label{bnd_Gtz}
\end{equation*}
 The bound of Gasch\"{u}tz et al. was further generalized by Vermani in 1969 \cite{Vermani1}. He obtained their result as a corollary of the bound 
 \begin{equation*} 
 |M(G)| \leq \bigg{|}M\bigg(\frac{G}{\gamma_c(G)}\bigg)\bigg{|}\bigg{|}Hom\bigg(\frac{G}{Z_{c-1}(G)}, \gamma_c(G)\bigg)\bigg{|}\bigg{/}|\gamma_c(G)|, \label{bnd_vrmni1}
 \end{equation*}
where $c$ is the nilpotency class of $G$. This bound was reproved by Jones using a different method in \cite{Jones1}. In 1969 Green's bound was generalized by Weigold  \cite{Weigold2} when he proved that 
\begin{equation}
|M(G)| \leq p^{\frac{1}{2}(d(G)-1)(2n-d(G))}. \label{bnd_wgld1}
\end{equation} 
In 1972 Jones, generalizing Green's bound, proved in \cite{Jones2} that, if the exponent of the center is $p^{e_{Z(G)}}$, then 
\begin{equation}
|\gamma_2(G)||M(G)| \leq p^{\frac{1}{2}(n-e_{Z(G)})(n+e_{Z(G)}-1)}. \label{bnd_jns}
\end{equation}
 This bound of Jones was further generalised by Vermani in 1974 \cite{Vermani2}. He proved that if the restriction homomorphism from $M(G)$ to $M(K)$, for a central subgroup $K$, is zero (Note that for a cyclic central subgroup $K$, it is zero), then 
 \begin{equation}
|\gamma_2(G)||M(G)| \leq |(G/K)^{ab} \otimes K|p^{\frac{1}{2}(m-r)(m+r-1)}, \label{bnd_vrmni2}
\end{equation} 
where $m$ and $r$ are given by  $|G/K| = p^m$ and $|\gamma_2(G)K/K| = p^r.$

In 1999 Ellis and Weigold, sharpening Weigold's earlier bound \ref{bnd_wgld1}, proved that \begin{equation}
|M(G)| \leq p^{\frac{1}{2}(d(G)-1)(2n-m)} (= p^{\frac{1}{2}(d(G)-1)(n+k)}), \label{bnd_ellis}
\end{equation}
where $m$ and $k$ are given by $|G^{ab}| = p^m, |\gamma_2(G)| = p^k$ \cite{Ellis1}. 
 


\vspace{.2cm}



Using the bound \ref{bnd_jns} Jones derived the corollary that $|M(G)| \leq p^{\frac{1}{2}n(n-1)- k}.$  Vermani, using the bound of Gasch\"{u}tz et al., noticed in \cite[Proposition 2.2]{Vermani2} that 
 \[|M(G)| \leq p^{\frac{1}{2}(n-k-1)(n+k)} = p^{\frac{1}{2}n(n-1)-\frac{1}{2}k(k-1)} \ .\] 
 For non-abelian $p$-groups $G$, Niroomand proved that 
 \begin{equation}
 |M(G)| \leq p^{\frac{1}{2}(n-k-1)(n+k-2)+1}, \label{bnd_nrmnd2}
 \end{equation}
where $k$ is given by $|\gamma_2(G)| = p^k$ \cite{Niroomand2}. Further, in another paper Niroomand and Russo proved that the bound \ref{bnd_nrmnd2} is better than the bound \ref{bnd_ellis} of Ellis and Weigold provided $G^{ab}$ is elemementary abelian \cite[Theorem 1.2]{Niroomand1}. 

\vspace{.2cm}
We mention here that the bound \ref{bnd_ellis} of Ellis and Weigold  was derived from the following more general bound of theirs.
 
 \begin{equation}
|M(G)| \leq p^{\frac{1}{2}d(m-e) + (\delta-1)(n-m)-max(0, \delta-2)}  \label{bnd_ellis1}
\end{equation}
 where $m$ and $p^e$ are the order and the exponent of $G^{ab}$ respectively and $d$ and $\delta$ are the minimal no. of generators of $G$ and $G/Z(G)$ respectively.

We then note that the bound \ref{bnd_ellis1} of Ellis and Weigold is more general than the bound \ref{bnd_nrmnd2} of Niroomand. In the following theorem we see that a visibly more general bound than the bound \ref{bnd_nrmnd2} can be derived from the bound. \ref{bnd_ellis1}. 

\begin{theorem}\label{thm1}
Let $G$ be a non-abelian finite $p$-group of order $p^n$ with $|\gamma_2(G)| =p^k$ and $d(G) = d$. Then
\[|M(G)| \leq p^{\frac{1}{2}(d-1)(n+k-2)+1}.\]
\end{theorem}


\vspace{.2cm}

Ellis and Weigold also noticed that their bound \ref{bnd_ellis} is attained if $G = C_{p^e} \times C_{p^e} \times \ldots  \times C_{p^e}$.  The bound \ref{bnd_ellis} of Ellis and Weigold was rederived by Niroomand and Russo. They further improved the bound when $G^{ab} \neq C_{p^e} \times C_{p^e} \times \ldots  \times C_{p^e}$ proving that 
\begin{equation}
|M(G)| \leq  p^{\frac{1}{2}(d(G)-1)(n+k-1)} \label{bnd_nrmnd1} 
\end{equation}
in this case.

Now of course Theorem \ref{thm1} provides a visibly stronger bound than the bound \ref{bnd_nrmnd1}. But the bound \ref{bnd_nrmnd1} motivates us to investigate further the case $G^{ab} \neq C_{p^e} \times C_{p^e} \times \ldots  \times C_{p^e}$. The following theorem sharpens the bound \ref{bnd_ellis1}. 

\begin{theorem} \label{thm4}
Let $G$ be a finite $p$-group of order $p^n$ with $d(G) = d, d(G/Z(G) = \delta$ and $G^{ab} = C_{p^{\alpha_1}} \times C_{p^{\alpha_2}} \times \cdots \times C_{p^{\alpha_d}} (\alpha_1 \geq \alpha_2 \geq \cdots \geq \alpha_d)$. Then 
\[|M(G)| \leq p^{\frac{1}{2}(d-1)(n-k-(\alpha_1-\alpha_d)) + (\delta-1)k-max(0, \delta-2)}.\]
\end{theorem}

To see that the bound in the above corollary is better than the bound \ref{bnd_ellis1}, we divide the right hand side of the bound \ref{bnd_ellis1} by the right hand side of the above bound and get the value $p^{\frac{1}{2}[(d-1)(\alpha_1-\alpha_d)-d\alpha_1+n-k]}$ 
which equals $p^{\frac{1}{2}[-(d-1)\alpha_d-\alpha_1+n-k]}$. 
Which on putting $n-k = \alpha_1+\alpha_2+ \cdots+ \alpha_d$ becomes 
$p^{\frac{1}{2}[\alpha_2+\cdots+\alpha_d-(d-1)\alpha_d)]}$. 
But this value is clearly greater than or equal to 1 because $\alpha_1 \geq \alpha_2 \geq \cdots \geq \alpha_d$.

\vspace{.2cm}
As a consequence we derive the following corollary.

\begin{corollary}\label{cor1}
Let $G$ be a non-abelian finite $p$-group of order $p^n$ with $d(G) = d$, $|\gamma_2(G)| = p^k$ and $G^{ab} = C_{p^{\alpha_1}} \times C_{p^{\alpha_2}} \times \cdots \times C_{p^{\alpha_d}} (\alpha_1 \geq \alpha_2 \geq \cdots \geq \alpha_d)$. Then  
\[|M(G)| \leq  p^{\frac{1}{2}(d-1)(n+k-2-(\alpha_1-\alpha_d))+1}.\]
In particular if $G^{ab}$ is not homocyclic, then
\[|M(G)| \leq  p^{\frac{1}{2}(d-1)(n+k-3)+1}.\]

\end{corollary}

\vspace{.4cm}
The bound \ref{bnd_vrmni2} of Vermani comes with a hypothesis, so it can not be compared in general with the bound obtained in Theorem \ref{thm1}. Though using Theorem \ref{thm1}  we can improve the bound of Vermani for non-abelian finite $p$-groups of nilpotency class at least 3.

\begin{theorem}\label{thm2}
Let $G$ be a finite $p$-group of nilpotency class at least 3 and $K$ a central subgroup of $G$ such that the restriction homomorphism from $M(G)$ to $M(K)$ is zero. Also assume that $|G/K| = p^m$ and $|\gamma_2(G)K/K| = p^r$. Then 
\[|\gamma_2(G)||M(G)| \leq \big|(G/K)^{ab} \otimes K\big|p^{\frac{1}{2}d(G/K)(m+r-2)+1}.\]
In particular, 
\[|\gamma_2(G)||M(G)| \leq \big|(G/K)^{ab} \otimes K\big|p^{\frac{1}{2}(m-r)(m+r-2)+1}.\]
\end{theorem}

By the coclass of a $p$-group $G$ of order $p^n$ we mean the number $n-c$ where $c$ is the nilpotency class of $G$. In 2009, Moravec proved, for finite $p$-group $G$ of coclass $r$, that $|M(G)| \leq  p^{r^2+(k+2)r}$ where $k$ is given by $|\gamma_2(G)| = p^k$ \cite[Theorem 1.1]{Moravec1}. The following Theorem improves this bound.

\begin{theorem} \label{thm3}
Let $G$ be a finite $p$-group of order $p^n$ and coclass $r$ with $|\gamma_2(G)| = p^k$. Then $|M(G)| \leq p^{\frac{1}{2}(r^2-r)+kr+1}$.  
\end{theorem}

\section{Proofs of Theorems}

{\bf Proof of Theorem \ref{thm1}}
Since $G$ is non-abelian we have $d(G/Z(G)) = \delta \geq 2$, otherwise $G/Z(G)$ is cyclic and $G$ is abelian. Therefore $max(0, \delta-2) = \delta-2$. Let $e$ be the exponent of $G^{ab}$. Then from the bound \ref{bnd_ellis1} we get

\begin{eqnarray*} \label{derivation_nrmnd}
|M(G)| & \leq & p^{\frac{1}{2}d(n-k-e) + (\delta-1)k-max(0, \delta-2)} \nonumber \\
        & = & p^{\frac{1}{2}(d-1)(n-k-e)+\frac{1}{2}(n-k-e) + (\delta-1)k- (\delta-2)} \nonumber \\
        & = & p^{\frac{1}{2}(d-1)(n-k-e) +\frac{1}{2}(n-k-e) + (\delta-1)(k-1)+1} \nonumber \\
       & = & p^{\frac{1}{2}(d-1)(n-k) - \frac{1}{2}(d-1)e +\frac{1}{2}(n-k-e) + (d-1)(k-1) - (d-\delta)(k-1)+1} \nonumber \\
    & = & p^{\frac{1}{2}(d-1)(n-k+2k-2) - \frac{1}{2}(d-1)e +\frac{1}{2}(n-k-e) - (d-\delta)(k-1)+1} \nonumber \\
        & = & p^{\frac{1}{2}(d-1)(n+k-2)+1 - [\frac{1}{2}(d-1)e -\frac{1}{2}(n-k-e) + (d-\delta)(k-1)]} \nonumber 
\end{eqnarray*}

Now notice that $\frac{1}{2}(d-1)e -\frac{1}{2}(n-k-e)$ is a non-negative value. To see this, let 
$G^{ab} = C_{p^{\alpha_1}} \times C_{p^{\alpha_2}} \times \cdots \times C_{p^{\alpha_d}} (\alpha_1 \geq \alpha_2 \geq \cdots \geq \alpha_d)$. Then $n-k = \alpha_1+\alpha_2+\cdots + \alpha_d$ and $e = \alpha_1$. Therefore
\begin{eqnarray*}
\frac{1}{2}(d-1)e -\frac{1}{2}(n-k-e) &=& \frac{1}{2}(d-1)\alpha_1 -\frac{1}{2}(\alpha_1+\alpha_2+\cdots + \alpha_d-\alpha_1)\\
&=& \frac{1}{2}\bigg(\alpha_1 +\alpha_1 +\cdots \alpha_1 \ \ (d-1 \ \text{times})\bigg) -\frac{1}{2}(\alpha_2+\cdots + \alpha_d)\\
&=& \frac{1}{2}(\alpha_1 -\alpha_2) + (\alpha_1-\alpha_3) +\cdots (\alpha_1-\alpha_d),
\end{eqnarray*}
which is clearly a non-negative value because $\alpha_1 \geq \alpha_2 \geq \cdots \geq \alpha_d$. Obviously $(d-\delta)(k-1)$ is a non-negative value. Hence we have 
\begin{equation*}
\boldsymbol{ |M(G)| \leq p^{\frac{1}{2}(d-1)(n+k-2)+1}}.
\end{equation*}

The following Lemma sharpens the bound \ref{bnd_nrmnd1} for abelian $p$-groups. 

\begin{lemma}\label{lem1}
Let $G$ be an abelian $p$-group of order $p^n$ such that $G = C_{p^{\alpha_1}} \times C_{p^{\alpha_2}} \times \cdots \times C_{p^{\alpha_d}} (\alpha_1 \geq \alpha_2 \geq \cdots \geq \alpha_d)$ and $|G| = p^n$, then $|M(G)| \leq p^{\frac{1}{2}(d(G)-1)(n-(\alpha_1-\alpha_d))}$.
\end{lemma}

\begin{proof}
For $d =2$, it is obvious from \cite[Corollary 2.2.12]{Karpilowski}. So let us assume that $d \geq 3$. Let $\alpha_i = \frac{n-\alpha_1-\alpha_d}{(d-2)}+k_i$ for $i = 2,3, \cdots, d-1$. Then by \cite[Corollary 2.2.12]{Karpilowski} 
\begin{eqnarray*}
|M(G)| & = & p^{\alpha_2+ 2\alpha_3+ \cdots +(d-1)\alpha_d}\\
       &=&  p^{(d-1)(\alpha_2+ \alpha_3 + \cdots +\alpha_d) - (d-2)\alpha_2 -(d-3)\alpha_3 -\cdots- \alpha_{d-1}}\\
       &=& p^{(d-1)(n-\alpha_1) - (d-2)\alpha_2 -(d-3)\alpha_3 -\cdots- \alpha_{d-1}}\\
       &=& p^{(d-1)(n-\alpha_1) - \frac{(n-\alpha_1-\alpha_d)}{d-2}\Big(1+2+ \cdots d-2\Big)-\Big[(d-2)k_2+(d-3)k_3+ \cdots + k_{d-1}\Big]}\\  
       &=& p^{(d-1)(n-\alpha_1) - \frac{1}{2}(n-\alpha_1-\alpha_d)(d-1)-\Big[(d-2)k_2+(d-3)k_3+ \cdots + k_{d-1}\Big]}.\\
\end{eqnarray*}

Now observe that $k_2 + k_3 + \cdots + k_{d-1} = 0$. Also notice that there exist a $j$ such that $k_2, k_3, \cdots, k_j$ are non-negative values and $k_{j+1}, k_{j+2}, \cdots, k_{d-1}$ are non-positive values. It follows that the value $(d-2)k_2+(d-3)k_3+ \cdots + k_{d-1}$ is non-negative. Therefore, $|M(G)| \leq p^{\frac{1}{2}(d(G)-1)(n-(\alpha_1-\alpha_d))}$.       
\end{proof}

{\bf Proof of Theorem \ref{thm4}}
Let $\Psi$ be as defined in \cite[Proposition 1]{Ellis1}. Following \cite[Proposition 1]{Ellis1} we have that 
\[|M(G)||\gamma_2(G)||\text{Im} \Psi| \leq |M(G^{ab})|p^{k\delta}.\]
But note from \cite[Proposition 1]{Ellis1} that $|\text{Im} \Psi| \geq p^{max(0, \delta-2)}$. Now the theorem follows from the Lemma \ref{lem1}.

\vspace{.2cm}

{\bf Proof of Corollary \ref{cor1}}
Having Theorem \ref{thm4} in hand the proof of the corollary runs on the same lines as the proof of Theorem \ref{thm1}.




\vspace{.4cm}
{\bf Proof of Theorem \ref{thm2}:}
By \cite[Theorem 1.2]{Vermani2} we have that 
\[|\gamma_2(G)||M(G)| \leq |M(G/K)||G/\gamma_2(G)K \otimes K|\bigg|\frac{\gamma_2(G)}{\gamma_2(G) \cap K}\bigg|.\]
Since $G$ is of nilpotency class at least 3, $G/K$ is non-abelian. Hence $m \geq r+2$. Applying Theorem \ref{thm1} we get that
\begin{eqnarray*}
|\gamma_2(G)||M(G)| & \leq & |G/\gamma_2(G)K \otimes K|p^{\frac{1}{2}(d(G/K)-1)(m+r-2)+1+r}\\
 & = & |G/\gamma_2(G)K \otimes K|p^{\frac{1}{2}d(G/K)(m+r-2)-\frac{1}{2}(m+r-2)+1+r}\\
 & \leq & |G/\gamma_2(G)K \otimes K|p^{\frac{1}{2}d(G/K)(m+r-2)-\frac{1}{2}(r+2+r-2)+1+r}\\
 & \leq & |G/\gamma_2(G)K \otimes K|p^{\frac{1}{2}d(G/K)(m+r-2)+1}.\\
 \end{eqnarray*}
 
This proves the theorem.

\vspace{.6cm}
{\bf Proof of Theorem \ref{thm3}:} 
Let $c$ be the nilpotency class of $G$. It is obvious that $c \leq k+1$ and $d(G) \leq n-k$. Therefore $d(G) \leq n-c+1 = r+1$, so that $d(G) -1 \leq r$. 
The inequality $c \leq k+1$ can be written as $n-(n-c) \leq k+1$, i.e., $n-r \leq k+1$ so that $n+k-2 \leq r+2k-1$. Using Theorem \ref{thm1} with the inequalities $d(G) -1 \leq r$ and $n+k-2 \leq r+2k-1$ we get the required result.


\vspace{.3cm}

{\bf Acknowledgements:} I am very grateful to my post-doctoral superviser Prof. Boris Kunyavski\u{\i} for his encouragement and support. I am also thankful to Bar-Ilan University, Israel for providing excelent research facilities.

\end{document}